\documentclass[12pt]{amsart}
\usepackage{fullpage,url,amssymb,enumerate,colonequals}

\usepackage{mathrsfs} 
\usepackage[justification=centering]{caption}

\makeatletter
\renewcommand\part{%
   \if@noskipsec \leavevmode \fi
   \par
   \addvspace{4ex}%
   \@afterindentfalse
   \secdef\@part\@spart}

\def\@part[#1]#2{%
    \ifnum \c@secnumdepth >\m@ne
      \refstepcounter{part}%
      \addcontentsline{toc}{part}{\thepart\hspace{1em}#1}%
    \else
      \addcontentsline{toc}{part}{#1}%
    \fi
    {\parindent \z@ \raggedright
     \interlinepenalty \@M
     \normalfont
     \ifnum \c@secnumdepth >\m@ne
       \large\bfseries \partname\nobreakspace\thepart.
     \fi
     \large \bfseries #2%
     \par}%
    \nobreak
    \vskip 3ex
    \@afterheading}
\def\@spart#1{%
    {\parindent \z@ \raggedright
     \interlinepenalty \@M
     \normalfont
     \large \bfseries #1\par}%
     \nobreak
     \vskip 3ex
     \@afterheading}
\makeatother

\usepackage{color}

\usepackage{xr-hyper}

\usepackage[
        colorlinks, citecolor=darkgreen,
        backref,
        pdfauthor={Filip Najman}, 
]{hyperref}
\usepackage{cleveref}
\usepackage{comment}

\newcommand{\F}{\mathbb{F}}

\newcommand{\Q}{\mathbb{Q}}

\newcommand{\Z}{\mathbb{Z}}

\newcommand{\rhobar}{{\overline{\rho}}}

\newcommand{\GL}{\operatorname{GL}}

\newcommand{\PGL}{\operatorname{PGL}}

\newcommand{\SL}{\operatorname{SL}}

\numberwithin{equation}{section}

\newtheorem{lemma}[equation]{Lemma}

\newtheorem{proposition}[equation]{Proposition}

\theoremstyle{definition}

\theoremstyle{remark}
\newtheorem{remark}[equation]{Remark}

\definecolor{darkgreen}{rgb}{0,0.5,0}

\makeatletter
\let\@wraptoccontribs\wraptoccontribs
\makeatother

\begin{document}

\title[On $r$-isogenies over $\mathbb{Q}(\zeta_r)$]{On $r$-isogenies over $\mathbb{Q}(\zeta_r)$ \\ of elliptic curves with rational $j$-invariants}

\author{Filip Najman}
\address{University of Zagreb, Bijeni\v{c}ka Cesta 30, 10000 Zagreb, Croatia}
\email{fnajman@math.hr}

\date{\today}


\begin{abstract}
The main goal of this paper is to determine for which odd prime numbers~$r$ can an elliptic curve~$E$ defined over $\Q$ have an $r$-isogeny over $\Q(\zeta_r)$. We study this question under various assumptions on the 2-torsion of $E$. Apart from being a natural question itself, the mod~$r$ representations attached to such $E$ arise in the Darmon program for the generalized Fermat equation of signature~\((r,r,p)\), playing a key role in the proof of modularity of certain Frey varieties in the recent work of Billerey, Chen, Dieulefait and Freitas.
\end{abstract}

\thanks{I gratefully acknowledge support by QuantiXLie Centre of Excellence, a project co-
financed by the Croatian Government and European Union through the European Regional Devel-
opment Fund - the Competitiveness and Cohesion Operational Programme (Grants KK.01.1.1.01.0004
and PK.1.1.02) and by the Croatian Science Foundation under the project no. IP-2022-10-5008.}

\maketitle

{
\hypersetup{linkcolor=black}
}

\section{Introduction}

In their recent work~\cite{BCDF3}, Billerey, Chen, Dieulefait and Freitas develop an approach to the generalized Fermat equation of signature~\((r,r,p)\) based on ideas from Darmon's program \cite{DarmonDuke} and a construction of Frey hyperelliptic curves due to Kraus. The results of the present paper play a key role in their proof of the modularity of the Jacobians of these Frey hyperelliptic curves and other abelian varieties considered by Darmon.

Let~\(r\geq 3\) be a prime number. For an elliptic curve~\(E\) defined over a number field~\(K\), we denote by~\(\rhobar_{E,r} : G_K \to \GL_2(\F_r)\) its mod~$r$ Galois representation (after fixing a basis for the $r$-torsion module~$E[r]$). The present paper is concerned with the (ir)reducibility of the representation~\(\rhobar_{E,r}\) where~\(K = \Q(\zeta_r)\) and~\(E\) is a base change to~\(K\) of an elliptic curve defined over~\(\Q\). 

In some cases our results will also hold for certain elliptic curves $E/\Q(\zeta_r)$ with $j(E)\in \Q$, not just those which are base changes of elliptic curves defined over $\Q$; in those instances we will state our results in this greater generality.

The reducibility of~\(\rhobar_{E,r} : G_{\Q(\zeta_r)}\to \GL_2(\F_r)\) is equivalent to~\(E/\Q(\zeta_r)\) having an isogeny of degree~\(r\) over~\(\Q(\zeta_r)\). We study this question under various assumptions on the \(2\)-torsion of~\(E\).

All code for Magma computations used to verify the claims in the paper can be found at
\begin{center}
{\href{https://github.com/F-Najman/r-isogenies}{\url{https://github.com/F-Najman/r-isogenies}}}.
\end{center}

\section*{Acknowledgements}

I thank Nicolas Billerey and Nuno Freitas for asking questions motivating this paper and for many helpful comments, Davide Lombardo for pointing out the proof of \Cref{prop1} for $r<17$, and the anonymous referees for their many helpful suggestions.

\section{Elliptic curves with $r$-isogenies over $\mathbb{Q}(\zeta_r)$}

In this subsection we determine when do elliptic curves $E$ defined over $\Q$, or in some instances with $j(E)\in \Q$, have an $r$-isogeny over $\Q(\zeta_r)$. Having an $r$-isogeny over $\Q(\zeta_r)$ is equivalent to $\rhobar_{E,r}(G_{\Q(\zeta_r)})$ being conjugate in $\GL_2(\F_r)$ to a subgroup of a Borel subgroup.

For a subgroup $G\leq \GL_2(\F_r)$, define $S(G):=G\cap \SL_2(\F_r)$. Denote by $C_s(r)$ the subgroup of diagonal matrices in $\GL_2(\F_r)$ and by $C_s^+(r)$ its normalizer in $\GL_2(\F_r)$. Let $\epsilon=-1$ if $r\equiv 3 \pmod 4$ and otherwise let $\epsilon\geq 2$  be the smallest integer which is not a quadratic residue modulo $r$. Denote by $C_{ns}(r)$ the group
$$C_{ns}(r):=\left\{ \begin{pmatrix}
  a & \epsilon b \\
  b & a
\end{pmatrix},(a, b) \in \F_r^2\backslash \{(0,0)\}\right\},$$
whose conjugates in $\GL_2(\F_r)$ are the non-split Cartan subgroups and let

\begin{equation}
\label{cns+def}
C_{ns}^+(r):=\left\{
\begin{pmatrix}
  a & \epsilon b \\
  b & a
\end{pmatrix}, (a, b) \in \F_r^2\backslash \{(0,0)\} \right\}
\cup
\left\{
 \begin{pmatrix}
  c & \epsilon d \\
  -d & -c
\end{pmatrix}, (c, d) \in \F_r^2\backslash \{(0,0)\} \right\}
\end{equation}
be the normalizer of $C_{ns}(r)$ in $\GL_2(\F_r)$.

\begin{lemma}\label{lem:subgroups}
Let $r$ be an odd prime. The group $S(C_s^+(r))$ is conjugate in~\(\GL_2(\F_r)\) to
 $$\left\{ \begin{pmatrix}
  a & 0 \\
  0 & a^{-1}
\end{pmatrix},\begin{pmatrix}
  0 & a \\
  -a^{-1} & 0

\end{pmatrix}, a \in \F_r^\times \right\}.$$

The group $S(C_{ns}^+(r))$ is the subset of matrices in \eqref{cns+def} with determinant $1$. The groups $S(C_s^+(r))$ and $S(C_{ns}^+(r))$ are not conjugate in $\GL_2(\F_r)$ to a subgroup of a Borel subgroup.

\end{lemma}
\begin{proof}
It is straightforward to check that $S(C_s^+(r))$ and $S(C_{ns}^+(r))$ are of the claimed form and that $S(C_s^+(r))$ acts freely on $\F_r^2 \backslash \{(0,0)\}$, from which it follows that the length of each orbit is $2(r-1)$. As a fixed $1$-dimensional subspace would give an orbit of length $\leq r - 1$, it follows that there are no fixed $1$-dimensional subspaces of $\F_r^2$.

The group $S(C_{ns}^+(r))$ has order $2(r+1)$ since $\det :C_{ns}^+(r) \rightarrow\F_r^\times$ is onto. Examine the action of $S(C_{ns}^+(r))$ on $\F_r^2 \backslash \{(0,0)\}$. The characteristic polynomial of a matrix in $S(C_{ns}^+(r))$ is either $x^2-2ax+1$ or $x^2+1$ and hence $1$ is a root of it only in the first case when $a=1$, which happens only for the identity matrix. We conclude that $S(C_{ns}^+(r))$ acts freely on $\F_r^2 \backslash \{(0,0)\}$ and hence the orbit of any $(x,y)$ has cardinality $2(r+1)$ and again we can conclude that there are no fixed $1$-dimensional subspaces of $\F_r^2$.

\end{proof}

We will need the following result of Zywina (\cite[Proposition 1.13.]{zyw}); see~\cite[Appendix~B]{Le_Fourn_Lemos} for a published proof.

\begin{proposition}[Zywina]
\label{prop-zyw}
Suppose that $E/\Q$ does not have CM, $r\geq 17$, $(r,j(E))\notin\{ (17,-17\cdot 373^3/2^{17}), (17,-17^2\cdot 101^3/2), (37,-7\cdot 11^3), (37,-7\cdot 137^3\cdot 2083^3)\}$ and $\rhobar_{E,r}$ is not surjective. Then
\begin{enumerate}
\item If $r\equiv 1 \pmod 3$, then $\rhobar_{E,r}(G_\Q)$ is conjugate in $\GL_2(\F_r)$ to $C_{ns}^+(r)$.
\item If $r\equiv 2 \pmod 3$, then $\rhobar_{E,r}(G_\Q)$ is conjugate in $\GL_2(\F_r)$ to either $C_{ns}^+(r)$ or
    $$G_3(r):=\left\{a^3,\ a \in C_{ns}(r)\right\} \cup \left\{\begin{pmatrix}
                        1 &  0\\
                        0 & -1
                      \end{pmatrix}a^3,\ a \in C_{ns}(r)\right\}.$$
\end{enumerate}
\end{proposition}

\begin{remark}
The recent results of Furio and Lombardo \cite{FL} prove that if $r\geq 37$, then the case (2) in \Cref{prop-zyw} is not possible.
\end{remark}



Given an elliptic curve $E$ defined over~\(\Q(\zeta_r)\) with $j(E)\neq 0,1728$, having an $r$-isogeny and a certain amount of $2$-torsion is a property depending only on the $j$-invariant, so we can (and will in the proofs of the propositions below) suppose that $E$ is defined over $\Q$. This is justifiable because if $E$ is not defined over $\Q$ and $j(E)\in \Q \backslash \{0,1728\}$, then $E$ has a quadratic twist $E'$ defined over $\Q$ and $E$ and $E'$ will have the same $2$-torsion structure and isogenies of the same degrees over $\Q(\zeta_r)$, and so we can replace~$E$ by~$E'$ in the arguments.

A subgroup $G$ of $\GL_2(\F_r)$ is called \textit{applicable} (cf. \cite[Definition 2.1]{zyw}) if $\det G=\F_r^\times$ and $G$ contains an element of trace $0$ and determinant $-1$ (which is the image of complex conjugation). Note that $\rhobar_{E,r}(G_\Q)$ is always an applicable subgroup \cite[Proposition 2.2]{zyw}.

\begin{proposition}
\label{prop1}
Let $r$ be an odd prime, $E/\Q(\zeta_r)$ with $j(E)\in \Q$ be an elliptic curve without complex multiplication (CM), such that $E$ has an $r$-isogeny over $\Q(\zeta_r)$ and
\begin{multline}\label{eq:execptions}
(r,j(E))\notin\left\{ (7,\frac{3^3\cdot 5 \cdot 7^5}{2^7})\right\}\\
 \cup \left\{\left(5,
\frac{5^4t^3(t^2 + 5t + 10)^3(2t^2 + 5t + 5)^3(4t^4 + 30t^3 + 95t^2 + 150t + 100)^3}{(t^2 + 5t + 5)^5 (t^4 + 5t^3 + 15t^2 + 25t + 25)^5}\right),  t\in \Q \right\}.
\end{multline}
Then $E$ has an $r$-isogeny over $\Q$.
\end{proposition}

\begin{proof}
As explained above we can assume that $E$ is defined over $\Q$. Suppose first $r\geq 17$.  We have shown in Lemma \ref{lem:subgroups} that $S(C_{ns}^+(r))$ does not fix any $1$-dimensional subspace of $\F_r^2$. Assume now $r\equiv 2 \pmod 3$. The group $\{a^3, a\in C_{ns}(r)\}$ is cyclic of order $\frac{r^2-1}{3}$, and the map $t\mapsto t^3$ is an automorphism of $\F_r^\times$, so we conclude that $\#S(G_3(r))=\frac{2(r+1)}{3}$.

Now we study the action of $S(G_3(r))$ on $\F_r^2 \backslash \{(0,0)\}$. Since $S(G_3(r))\leq S(C_{ns}^+(r))$ this action is free. Hence the orbit of any $(x,y)$ has cardinality $2(r+1)/3$. It is again easy to see that there cannot be a fixed $1$-dimensional subspace of $\F_r^2$, as otherwise there would be orbits $S_1, \ldots, S_k$ in $\F_r^2 \backslash \{(0,0)\}$ such that the sum of their lengths is $r-1$. This is clearly impossible.
Now applying Proposition \ref{prop-zyw} completes the proof for $r\geq 17$.

To deal with $r< 17$ we note that $E$ has an $r$-isogeny over $\Q(\zeta_r)$ but no $r$-isogeny over $\Q$ if and only if $\rhobar_{E,r}(G_\Q)$ acts irreducibly on $\F_r^2$, but $\rhobar_{E,r}(G_\Q)\cap \SL_2(\F_r)$ acts reducibly on $\F_r^2$. A brute force search among all proper applicable subgroups of $\GL_2(\F_r)$ yields, up to conjugacy, one possible subgroup for $r=5$, 2 groups for $r=7,11$ and $13$ each.

For $r=13$ the candidate groups are of index 182 and 546, which are not possible by \cite[Theorem 1.8 and Remark 1.9]{zyw} and \cite[Theorem 1.1.]{bdmtv2}. For $r=11$ the candidate groups are of index 132 and 264, which are not possible by \cite[Theorem 1.6]{zyw}.

For $r=7$, the two candidate groups are of index 56 and 112, and both are possible by \cite[Theorem 1.5]{zyw}; they are conjugate in $\GL_2(\F_7)$ to the groups denoted $G_1$ and $H_{1,1}$ in \cite[Section 1.4]{zyw}. As $G_1=\pm H_{1,1}$ and our conditions are quadratic-twist invariant, by \cite[Theorem 1.5]{zyw}, it follows that $E$ has a $7$-isogeny over $\Q(\zeta_7)$, but no $7$-isogeny over $\Q$, if and only if $j(E)=\frac{3^3\cdot 5 \cdot 7^5}{2^7}$.

For $r=5$ the only candidate group we find is conjugate in $\GL_2(\F_5)$ to the group denoted $G_3$ in \cite[Section 1.3]{zyw}. By \cite[Theorem 1.4. (ii)]{zyw}, $\rhobar_{E,3}\subseteq G_3$ if and only if $j(E)$ is of the form given in \eqref{eq:execptions}.

\end{proof}

\begin{remark}
  Recall (see e.g. \cite[Table 4]{lr}) that for $r\geq 17$, $E/\Q$ without CM has an $r$-isogeny if and only if $$(r,j(E))\in\left\{ (17,-17\cdot 373^3/2^{17}), (17,-17^2\cdot 101^3/2), \right.
\left.(37,-7\cdot 11^3), (37,-7\cdot 137^3\cdot 2083^3)\right\},$$ so by \Cref{prop1} these are the only instances when a non-CM elliptic curve defined over $\Q$ can have an $r$-isogeny over $\Q(\zeta_r)$ for $r\geq 17$.
\end{remark}

\begin{proposition}
\label{prop:ired}
Suppose that $r\geq 5$, that $E/\Q(\zeta_r)$ has $CM$ and either $j(E)\in \Q\backslash \{0,1728\}$ or if $j(E)\in\{0,1728\}$ that $E$ is a base change of an elliptic curve defined over $\Q$. The following are equivalent:
\begin{itemize}
  \item[(a)] The curve $E$ has an $r$-isogeny over $\Q(\zeta_r)$.
  \item[(b)] Every elliptic curve $E'/\Q$ with $j(E')=j(E)$ has an $r$-isogeny over $\Q$.
  \item[(c)] $r$ divides $D$, where $-D$ is the discriminant of the imaginary quadratic number field containing the endomorphism ring of $E$.
\end{itemize}
\end{proposition}
\begin{proof}

We can assume that $E=E'$, i.e., $E$ is a base change of an elliptic curve defined over $\Q$, using for $j(E)\notin\{0,1728\}$ the same argumentation as before \Cref{prop1}.

For $j(E)\neq 0$, this follows directly from \cite[Proposition 1.14]{zyw} and Lemma \ref{lem:subgroups}.

Suppose $j(E)=0$. By \cite[Proposition 1.16]{zyw} if $r\equiv 2,5, 8\pmod 9$ then $\overline \rho_{E,r}(G_\Q)$ contains the subgroup $G_3(r)$ defined in \Cref{prop-zyw}. Using the same argumentation as in the proof of Proposition \ref{prop1} we conclude that $E$ has no $r$-isogeny over $\Q(\zeta_r)$.

Suppose $r\equiv 1,4, 7\pmod 9$. By \cite[Proposition 1.16]{zyw} it follows that $\overline \rho_{E,r}(G_\Q)$ contains the subgroup $G$ of $C_s^+(r)$ consisting of matrices of the form $\begin{pmatrix}
                        a & 0 \\
                        0 & b
                      \end{pmatrix}$
and $\begin{pmatrix}
                        0 & a \\
                        b & 0
                      \end{pmatrix}$
with $a/b\in (\F_r^\times)^3$. Suppose that $\overline \rho_{E,r}(G_{\Q(\zeta_r)})$ is reducible, so $S(G)$ fixes a 1-dimensional subspace of $\F_r^2$ generated by $(x,y)$. Let $t\in \F_r$ such that $t^3 \neq t^{-3}$; such a $t$ exists for all $r\neq 7$. Then for
$B:=\begin{pmatrix}
                        t^3 & 0 \\
                        0 & t^{-3}
                      \end{pmatrix} \in S(G)$, since we have $B(x,y)=(t^3x,t^{-3}y)=\beta(x,y)$ for some $\beta \in \F_r^\times$ we conclude that either $x=0$ or $y=0$.

Since $A=\begin{pmatrix}
                        0 & 1 \\
                        -1 & 0
                      \end{pmatrix} \in S(G)$ we have $A(x,y)=(y,-x)=\alpha (x,y)$ for some $\alpha \in \F_r^\times$, so $y=\alpha x=-\alpha^2 y$, which implies that if one of $x$ or $y$ is equal to 0, then so is the other. The equality $y=\alpha x=-\alpha^2 y$ also eliminates the case $r=7$ as $-1$ is not a square in $\F_7^\times$, so again we get $y=x=0$. This gives a contradiction and the result.
\end{proof}

\section{Elliptic curves with a $2$-torsion point and reducible mod $r$ Galois representations}

In this section we prove results about surjectivity of mod $r$ Galois representations of elliptic curves $E/\Q$ with a point of order $2$ over $\Q$. 

\begin{proposition}\label{prop:B1}
Let~\(E/\Q\) be a non-CM elliptic curve such that \(E\) has a \(\Q\)-rational \(2\)-torsion point and let~\(r\ge11\) be a prime number. Then, we have~\(\rhobar_{E,r}(G_\Q) = \GL_2(\F_r)\). In particular, \(E\) has no isogeny of degree~\(r\) over~\(\Q(\zeta_r)\).
\end{proposition}
\begin{proof}
Denote by~\(j(E)\) the \(j\)-invariant of~\(E\) and assume for a contradiction that the representation~\(\rhobar_{E,r}:G_\Q\rightarrow\GL_2(\F_r)\) is not surjective. We first deal with the case~\(r = 13\). Recall that we are in one of the following situations:
\begin{enumerate}
\item\label{item:Borel} \(\rhobar_{E,13}(G_\Q)\) is included in a Borel subgroup of~\(\GL_2(\F_{13})\);
\item\label{item:NormC} \(\rhobar_{E,13}(G_\Q)\) is included in the normalizer of a Cartan subgroup of~\(\GL_2(\F_{13})\);
\item\label{item:excep} the projective image~\(H_{E,13}\) of~\(\rhobar_{E,13}(G_\Q)\) is isomorphic to~\(A_4\), \(S_4\), or~\(A_5\).
\end{enumerate}

We eliminate~(\ref{item:Borel}) with a result of Mazur-V\'elu \cite{MazurVelu}: If~\(E\) had a \(13\)-isogeny, then it would have a \(26\)-isogeny over~\(\Q\), which is impossible. The main result of~\cite{bdmtv} states that case~(\ref{item:NormC}) does not occur. Finally, \(H_{E,13}\) is not isomorphic to~\(A_5\) as the order of~\(\PGL_2(\F_{13})\) is not divisible by~\(5\). Therefore~\(H_{E,13}\subset S_4\) and according to~\cite[Theorem~1.1 and Section 5.1]{bdmtv2} (see also \cite[Corollary 1.9]{BanwaitCremona} ), we have that
\begin{multline*}
j(E)\in\left\{2^4\cdot 5\cdot 13^4\cdot 17^3/3^{13}, -2^{12}\cdot 5^3\cdot 11\cdot 13^4/3^{13},\right. \\
\left. 2^{18}\cdot 3^3\cdot 13^4\cdot 127^3\cdot 139^3\cdot 157^3\cdot 283^3\cdot 929/(5^{13}\cdot 61^{13})\right\}.
\end{multline*}
For each of these three values, we check that any elliptic curve with that $j$-invariant has trivial $2$-torsion over $\Q$, which shows that the case (\ref{item:excep}) is not possible. 

From now on, assume~\(r\ge11\) and~\(r\neq13\).  We first show that~\(\rhobar_{E,r}(G_\Q)\) is conjugate to a subgroup of~\(C_{ns}^+(r)\). For~\(r = 11\), this follows from~\cite[Theorem~1.6]{zyw} after checking that elliptic curves with $j$-invariant $-11^2$ and $-11\cdot 131^3$ have no non-trivial $2$-torsion over $\Q$. If~\(r\ge17\), this follows from Proposition~\ref{prop-zyw} (see also~\cite[Theorem~1.11]{zyw}), unless we have~\(j(E)\in\{ -17\cdot 373^3/2^{17}, -17^2\cdot 101^3/2, -7\cdot 11^3, -7\cdot 137^3\cdot 2083^3\}\). However, we check as before that these \(j\)-invariants correspond to elliptic curves with trivial rational \(2\)-torsion. Hence we have proved that~\(\rhobar_{E,r}(G_\Q)\) is conjugate to a subgroup of~\(C_{ns}^+(r)\).

According to~\cite[Proposition~2.1]{lemos}, we therefore have~\(j(E)\in\Z\). Since moreover \(E\) has a rational point of order~\(2\), its \(j\)-invariant~\(j(E)\) is of the shape
\[
j(E)=\frac{(t+16)^3}{t}
\]
for some~\(t\in\Q\) (\cite[p.~179]{Birch}) and hence one can see from \cite[p.~142]{lemos} (eliminating the CM \(j\)-invariants) that
\begin{multline*}
j(E) \in \{-2^2\cdot 7^3,-2^4\cdot 3^3,-2^6, 2^7, 2^4\cdot 5^3, 2^{11}, 2^2\cdot 3^6, 2^7\cdot 3^3, 17^3, 2^5\cdot 7^3, 2^5\cdot 3^6, 2^4\cdot 17^3, \\
2^3\cdot 31^3, 2^2\cdot 3^6\cdot 7^3, 2^2\cdot 5^3\cdot 13^3, 2\cdot 127^3, 2\cdot 3^3\cdot 43^3, 257^3\}.
     \end{multline*}
For each element in this list, we find an elliptic curve over~\(\Q\) with the corresponding \(j\)-invariant such that all of its mod $p$ representations are surjective for $p>3$ prime. We check this in LMFDB \cite{lmfdb}. In particular, its mod~\(r\) representation is surjective. Since the mod~\(r\) representation of a single quadratic twist is surjective if and only if it is surjective for all quadratic twists, the same conclusion holds for~\(E\). This gives the desired contradiction.

The last statement follows from the fact that \(\rhobar_{E,r}(G_{\Q(\zeta_r)})=\rhobar_{E,r}(G_{\Q}) \cap \SL_2(\F_r) = \SL_2(\F_r)\) has order greater than the order of a Borel subgroup of~\(\GL_2(\F_r)\).
\end{proof}

\begin{proposition} \label{prop7}
The only elliptic curves $E/\Q(\zeta_7)$ with $j(E)\in \Q$, a point of order $2$ and an isogeny of degree $7$ over $\Q(\zeta_7)$ are those with $j(E)\in \{-3^3\cdot5^3, 3^3\cdot5^3\cdot17^3\}$. Furthermore, $E/\Q(\zeta_7)$ with $j(E)\in \Q$ has full $2$-torsion over $\Q(\zeta_7)$ and a $7$-isogeny if and only if we have~$j(E)=-3^3\cdot5^3.$
\end{proposition}
\begin{proof}
  For $r=7$ we compute $X_0(14)(\Q(\zeta_7))\simeq \Z/2\Z\oplus \Z/6\Z$ ($X_0(14)$ is an elliptic curve). Of these $12$ points, 6 are rational; $4$ are cusps and $2$ correspond to elliptic curves over $\Q$ with $j$-invariants $3^{3} \cdot 5^{3} \cdot 17^{3}$ and $-3^{3} \cdot 5^{3}$.  Elliptic curves with either of these $j$-invariants have a single $2$-torsion point over $\Q$, multiplication by an order of $\Q(\sqrt{-7})$, and hence a $7$-isogeny over $\Q$. Recall that an elliptic curve with at least one $2$-torsion point has full $2$-torsion over a number field $k$ if and only if its discriminant is a square in $k$ and that taking a different model of the curve changes the discriminant by a 12th power, so having $2$-torsion is model-invariant.
  Hence, to check whether they acquire full $2$-torsion over $\Q(\zeta_7)$, it remains to check whether their discriminant is a square over $\Q(\zeta_7)$. For $j(E)=-3^3\cdot5^3$, we have $\Delta(E)\in -7 \cdot (\Q^\times)^2$ and for $j(E)=3^3\cdot5^3\cdot17^3$ we have $\Delta(E)\in 7 \cdot (\Q^\times)^2$. As $-7$ is a square in $\Q(\zeta_7)$ and $7$ is not, we conclude that only elliptic curves with $j(E)=-3^3\cdot5^3$ have full $2$-torsion and a $7$-isogeny over $\Q(\zeta_7)$.

  Of the remaining $6$ points which are not rational, $4$ correspond to $j$-invariants which are not $\Q$-rational, and $2$ more points correspond to the $j$-invariant $-3^3\cdot5^3$.  The fact that there are 3 points in $X_0(14)(\Q(\zeta_7))$ corresponding to the $j$-invariant $-3^3\cdot5^3$ is explained by the fact that an elliptic curve with $j$-invariant $-3^3\cdot5^3$ has full $2$-torsion over $\Q(\zeta_7)$ and hence 3 $G_{\Q(\zeta_7)}$-invariant subgroups of order $14$. 

\end{proof}

\begin{remark}
For $r=3$ and $5$ there exist infinitely many rational elliptic curves with a point of order $2$ and an isogeny of degree $r$, as the modular curves $X_0(6)$ and $X_0(10)$ have genus 0 and have rational cusps. Also, there are infinitely many elliptic curves with full $2$-torsion and reducible mod $3$ Galois representations; any elliptic curve with $\Z/2\Z \times \Z/6\Z$ torsion over $\Q$ is such an example.
\end{remark}

\section{Elliptic curves with full $2$-torsion and reducible mod $r$ Galois representations}

In the next two propositions we determine the elliptic curves with $j(E) \in \Q$ with both full $2$-torsion and an $r$-isogeny over $\Q(\zeta_r)$.

\begin{proposition}
\label{prop:CM}

Suppose that $r\geq 5$, that $E/\Q(\zeta_r)$ has CM and either $j(E)\in \Q\backslash \{0,1728\}$ or if $j(E)\in\{0,1728\}$ that $E$ is a base change of an elliptic curve defined over $\Q$. Suppose that $r$ divides the discriminant of the imaginary quadratic field containing the endomorphism ring of $E$ over $\overline \Q$. Then $E$ has full $2$-torsion and an $r$-isogeny over $\Q(\zeta_r)$ if and only if $r=7$ and $j(E)=-3^3\cdot5^3$.
\end{proposition}
\begin{proof}
We can assume that $E$ is defined over $\Q$, using for $j(E)\notin\{0,1728\}$ the same argumentation as before \Cref{prop1}.

By \Cref{prop:ired} it follows that if $E$ has an $r$-isogeny over $\Q(\zeta_r)$, then $E$ has an $r$-isogeny over $\Q$. If $\rhobar_{E,2}(G_\Q)=\GL_2(\F_2)$, then obviously $E$ does not have full $2$-torsion over $\Q(\zeta_r)$ (which is abelian over $\Q$).  Hence by \cite[Proposition 1.15]{zyw} we have
$$j(E)\in \left\{ 2^4\cdot3^3\cdot5^3, 2^3\cdot3^3\cdot 11^3, -3^3\cdot5^3, 3^3\cdot5^3\cdot17^3, 2^6\cdot 5^3\right\}.$$
By \Cref{prop:ired}, $E$ has an $r$-isogeny over $\Q$ if and only if $r$ divides $D$, where $-D$ is the discriminant of the imaginary quadratic number field containing the endomorphism ring of $E$. For $j(E)\in \left\{ 2^4\cdot3^3\cdot5^3, 2^3\cdot3^3\cdot 11^3, 2^6\cdot 5^3\right\},$ $D$ will be divisible by only $2$ or $3$ (see e.g. \cite[Table 1]{zyw}), so we need only consider elliptic curves with $j(E)=-3^3\cdot5^3$ and $3^3\cdot5^3\cdot17^3$. Now the result follows from \Cref{prop7}.


\end{proof}

\begin{proposition}
Suppose that $r\geq 5$, that $E/\Q(\zeta_r)$ and either $j(E)\in \Q\backslash \{0,1728\}$ or if $j(E)\in\{0,1728\}$ that $E$ is a base change of an elliptic curve defined over $\Q$. Then if $E$ has an $r$-isogeny and full $2$-torsion over $\Q(\zeta_r)$, then $j(E)=-3^3\cdot5^3$ and $r=7$.
\label{prop:full2irred}
\end{proposition}
\begin{proof}

As before we can assume that $E$ is a defined over $\Q$, using for $j(E)\notin\{0,1728\}$ the same argumentation as before \Cref{prop1}.

As we have already dealt with elliptic curves with CM in \Cref{prop:CM}, it remains to deal with elliptic curves without CM.

For $r\geq 17$, we check that for all elliptic curves with $j(E)\in  \{-17\cdot 373^3/2^{17}, -17^2\cdot 101^3/2, -7\cdot 11^3, -7\cdot 137^3\cdot 2083^3\}$, we have $\Delta(E)\in -10\cdot (\Q^\times)^2$ for the first 2 $j$-invariants, and $\Delta(E)\in -5\cdot (\Q^\times)^2$ for the last two $j$-invariants. Neither $\Q(\sqrt{-5})$ nor $\Q(\sqrt{-10})$ are contained in $\Q(\zeta_{17})$ or $\Q(\zeta_{37})$ and now the claim follows from \Cref{prop-zyw} and Lemma \ref{lem:subgroups}. \

For $r=5$ we compute that $X_0(20)(\Q(\zeta_5))=X_0(20)(\Q)$ ($X_0(20)$ is an elliptic curve). As an elliptic curve with full $2$-torsion and an $r$-isogeny would be $2$-isogenous to an elliptic curve with a $4r$-isogeny (see \cite[Lemma 7]{najman}), we are done with this case as $X_0(20)(\Q)$ consists of cusps \cite{ligozat}.

The case $r=7$ has already been dealt with in \Cref{prop7}.

For $r=11$, we compute $X_0(11)(\Q)=X_0(11)(\Q(\zeta_{11}))$, so the only elliptic curves with $11$-isogenies over $\Q(\zeta_{11})$ are those with $j(E) \in \{ -11^2, -2^{15}, -11\cdot131^3\}$. Elliptic curves with these $j$-invariants have trivial torsion over $\Q$, so cannot gain any $2$-torsion over a number field of degree not divisible by $3$.

Finally, for $r=13$, by \cite[Theorem 1.8.]{zyw}, the discussion after it and \cite[Theorem 1.1 and Theorem 1.3]{bdmtv}, we know that $\rhobar_{E,13}(G_\Q)$ for a non-CM elliptic curve has to be $\GL_2(\F_{13})$, contained in a conjugate of a Borel subgroup or conjugate to the group $G_7$ (using Zywina's notation)
$$G_7:=\left\langle\begin{pmatrix}
               2 & 0 \\
               0 & 2
             \end{pmatrix}, \begin{pmatrix}
               2 & 0 \\
               0 & 3
             \end{pmatrix}, \begin{pmatrix}
               0 & -1 \\
               1 & 0
             \end{pmatrix}, \begin{pmatrix}
               1 & 1 \\
               -1 & 1
             \end{pmatrix}\right\rangle.
             $$
We compute in Magma that the orbits of $S(G_7)$ acting on $\F_{13}^2 \backslash \{(0,0)\}$ are all of length 24, so elliptic curves with such image do not have $13$-isogenies over $\Q(\zeta_{13})$. We conclude that if a non-CM elliptic curve has a $13$-isogeny over $\Q(\zeta_{13})$, it has one already over $\Q$. As there are no $26$-isogenies over $\Q$ \cite{MazurVelu}, an elliptic curve with a $13$-isogeny over $\Q$ cannot have a 2-torsion point over $\Q$. Since $E$ gains full $2$-torsion over $\Q(\zeta_{13})$, it follows that $\Q(E[2])$ is a subfield of $\Q(\zeta_{13})$ and hence $\rhobar_{E,2}(G_\Q)$ is cyclic of order 3. But there are no such curves, which can be seen from \cite[Table A8]{DGJ} in the line [2Cn,13B]; this shows that the modular curve whose rational points parametrize elliptic curves $E/\Q$ having simultaneously $\rhobar_{E,2}(G_\Q)$ cyclic of order 3 (denoted 2Cn) and a $13$-isogeny (denoted 13B) has $2$ rational points, both of which can easily be seen to be cusps.

\end{proof}

For $r=5$ and $7$, if we assume that $E$ is an elliptic curve over $\Q$ with full $2$-torsion over $\Q$, we can prove a stronger result, similar to the one in \Cref{prop:B1}.

\begin{proposition}
\label{prop:57}
Let $E/ \Q$ be an elliptic curve without CM with full $2$-torsion over $\Q$. Then $\rhobar_{E,r}(G_\Q)=\GL_2(\F_{r})$ for $r=5,7$.
\end{proposition}
\begin{proof}
For $r=7$, by \cite[Theorem C]{Morrow}, if $\rhobar_{E,7}(G_\Q)\neq\GL_2(\F_{7})$, then $\rhobar_{E,7}(G_\Q)$ is contained in a Borel subgroup, which would imply that $E$ has a $14$-isogeny over $\Q$, which is impossible since it does not have CM.

For $r=5$, by \cite[Theorem C]{Morrow}, if $\rhobar_{E,5}(G_\Q)\neq\GL_2(\F_{5})$, then $\rhobar_{E,5}(G_\Q) \leq G_{9,5}$, where (using the notation of \cite{Morrow})
$$G_{9,5}:=\left \langle \begin{pmatrix}
                           2 & 0 \\
                           0 & 1
                         \end{pmatrix} , \begin{pmatrix}
                           1 & 0 \\
                           0 & 2
                         \end{pmatrix}, \begin{pmatrix}
                           0 & -1 \\
                           1 & 0
                         \end{pmatrix}, \begin{pmatrix}
                           1 & 1 \\
                           1 & -1
                         \end{pmatrix}\right \rangle\leq \GL_2(\F_5).$$
The projective image of the group $G_{9,5}$ is isomorphic to $S_4$, and is denoted by $5S4$ in \cite{DGJ}.

 In \cite[Table A13]{DGJ} the line $[2Cs,5S4]$ shows that the modular curve whose rational points parametrize elliptic curves $E/\Q$ having simultaneously $\rhobar_{E,2}(G_\Q)\{I\}$ (denoted $2Cs$) and having $\rhobar_{E,5}(G_\Q)\leq G_{9,5}$ (denoted $S_4$ as mentioned above) has 3 non-cuspidal rational points, all corresponding to $j(E)=2^6\cdot3^3$, and such curves have CM.
\end{proof}


\begin{thebibliography}{10}

\bibitem{bdmtv}
Jennifer Balakrishnan, Netan Dogra, J.~Steffen M\"{u}ller, Jan Tuitman, and Jan
  Vonk.
\newblock Explicit {C}habauty-{K}im for the split {C}artan modular curve of
  level 13.
\newblock {\em Ann. of Math. (2)}, 189(3):885--944, 2019.

\bibitem{bdmtv2}
Jennifer~S. Balakrishnan, Netan Dogra, J.~Steffen M\"{u}ller, Jan Tuitman, and
  Jan Vonk.
\newblock Quadratic {C}habauty for modular curves: algorithms and examples.
\newblock {\em Compos. Math.}, 159(6):1111--1152, 2023.

\bibitem{BanwaitCremona}
Barinder\thinspace{}S. Banwait and John\thinspace{}E. Cremona.
\newblock Tetrahedral elliptic curves and the local-global principle for
  isogenies.
\newblock {\em Algebra \& Number Theory}, 8(5):1201--1229, 2014.

\bibitem{BCDF3}
Nicolas~Billerey, Imin~Chen, Luis~Dieulefait, and Nuno~Freitas.
\newblock On {D}armon's program for the generalized {F}ermat equation, {I}.
preprint.

\bibitem{Birch}
Bryan\thinspace{}J. Birch.
\newblock Some calculations of modular relations.
\newblock In {\em Modular functions of one variable, {I} ({P}roc. {I}nternat.
  {S}ummer {S}chool, {U}niv. {A}ntwerp, 1972)}, pages 175--186. Lecture Notes
  in Mathematics, Vol. 320, 1973.

\bibitem{DGJ}
Harris~B. Daniels and Enrique Gonz\'{a}lez-Jim\'{e}nez.
\newblock Serre's constant of elliptic curves over the rationals.
\newblock {\em Exp. Math.}, 31(2):518--536, 2022.

\bibitem{DarmonDuke}
Henri Darmon.
\newblock Rigid local systems, {H}ilbert modular forms, and {F}ermat's {L}ast
  {T}heorem.
\newblock {\em Duke Math. J.}, 102(3):413--449, 2000.

\bibitem{FL}
Lorenzo Furio and Davide Lombardo.
\newblock Serre's uniformity question and proper subgroups of $C_{ns}^+(p)$
\newblock preprint (2023), \url{https://arxiv.org/abs/2305.17780}

\bibitem{Le_Fourn_Lemos}
Samuel Le~Fourn and Pedro Lemos.
\newblock Residual {G}alois representations of elliptic curves with image
  contained in the normaliser of a nonsplit {C}artan.
\newblock {\em Algebra Number Theory}, 15(3):747--771, 2021.

\bibitem{lemos}
Pedro Lemos.
\newblock Serre's uniformity conjecture for elliptic curves with rational
  cyclic isogenies.
\newblock {\em Trans. Amer. Math. Soc.}, 371(1):137--146, 2019.

\bibitem{ligozat}
Gerard {Ligozat}.
\newblock {Courbes modulaires de genre 1}.
\newblock {\em {Bull. Soc. Math. Fr., Suppl., M\'em.}}, 43:80, 1975.

\bibitem{lr}
Álvaro Lozano-Robledo.
\newblock{On the field of definition of
$p$-torsion points on elliptic curves over the rationals.}
\newblock {\em Math. Ann}. 357:279–-305, 2013.

\bibitem{lmfdb}
The {LMFDB Collaboration}.
\newblock The {L}-functions and modular forms database.
\newblock \url{http://www.lmfdb.org}, 2013.

\bibitem{MazurVelu}
Barry Mazur and Jacques V\'{e}lu.
\newblock Courbes de {W}eil de conducteur {$26$}.
\newblock {\em C. R. Acad. Sci. Paris S\'{e}r. A-B}, 275:A743--A745, 1972.

\bibitem{Morrow}
Jackson\thinspace{}S. Morrow.
\newblock Composite images of {G}alois for elliptic curves over {$\mathbb{Q}$} and
  entanglement fields.
\newblock {\em Math. Comp.}, 88(319):2389--2421, 2019.

\bibitem{najman}
Filip Najman.
\newblock Torsion of rational elliptic curves over cubic fields and sporadic
  points on {$X_1(n)$}.
\newblock {\em Math. Res. Lett.}, 23(1):245--272, 2016.

\bibitem{zyw}
David {Zywina}.
\newblock {On the possible images of the mod $\ell$ representations associated
  to elliptic curves over $\mathbb{Q}$}.
\newblock ArXiv preprint,
  \href{https://arxiv.org/abs/1508.07660}{arXiv:1508.07660}, 2015.

\end{thebibliography}
\end{document}